\documentclass[11pt]{amsart}

\usepackage{amssymb}
\usepackage{graphicx}
\usepackage{amsthm}
\usepackage{amscd}
\usepackage{mathtools}
\usepackage{amssymb}
\usepackage{amsmath}

\usepackage{graphicx}
\usepackage{pict2e}

\usepackage{array}
\usepackage{verbatim}
\usepackage{color}
\usepackage[utf8]{inputenc}

\newtheorem{corollary}{Corollary}
\newtheorem*{corollary*}{Korollar}
\newtheorem*{notation*}{Notation}

\newtheorem*{dank*}{Danksagung}
\newtheorem{lemma}{Lemma}
\newtheorem*{example*}{Example}
\newtheorem{theorem}{Theorem}
\newtheorem{proposition}{Proposition}
\theoremstyle{remark}
\newtheorem{remark}{Remark}
\theoremstyle{theorema}
\newtheorem*{theorema}{Theorem}
\newtheorem*{theoremb}{Acknowledgement}
\newtheorem{theoremc}{Problem}

\theoremstyle{definition}
\newtheorem{definition}{Definition}
\newtheorem*{notatio}{Notation}

\newcommand{\mathell}{\mathfrak{l}}

\title[Tits arrangements on cubic curves]
{Tits arrangements on cubic curves}

\author{Michael~Cuntz}
\address{Michael Cuntz, Leibniz Universit\"at Hannover,
Institut f\"ur Algebra, Zah\-lentheorie und Diskrete Mathematik,
Fakult\"at f\"ur Mathematik und Physik,
Wel\-fengarten 1,
D-30167 Hannover, Germany}
\email{cuntz@math.uni-hannover.de}

\author{David~Geis}
\address{David Geis, Leibniz Universit\"at Hannover,
Institut f\"ur Algebra, Zah\-len\-theorie und Diskrete Mathematik,
Fakult\"at f\"ur Mathematik und Physik,
Wel\-fengarten 1,
D-30167 Hannover, Germany}
\email{geis@math.uni-hannover.de}

\begin{document}

\begin{abstract}
We classify affine rank three Tits arrangements whose roots are contained in the locus of a homogeneous cubic polynomial.
We find that there exist irreducible affine Tits arrangements which are not locally spherical.
\end{abstract}

\maketitle

\section{Introduction}
Weyl groups are invariants of different types of algebras in Lie theory, \emph{Weyl groupoids} are invariants of (in a certain sense) more general quantum groups, the so-called Nichols algebras (see for example \cite{p-H-06}).
Geometrically, Weyl groupoids may be viewed as simplicial arrangements of hyperplanes which are not necessarily coming from a reflection group (see \cite{p-HW-10}, \cite{p-C10}). To go beyond the theory of finite dimensional Nichols algebras, it turns out that one needs an appropriate notion of infinite simplicial arrangement, which is the main contribution of \cite{CMW} where these are called \emph{Tits arrangements}.

In this paper we give a classification of \emph{affine rank three Tits arrangements} whose corresponding projective root vectors are contained in the locus of a homogeneous cubic polynomial. Our strategy for the classification builds upon the results obtained in \cite{CMW} and on elementary tools from the geometry of the projective plane, like Bézout's theorem and the fact that the conic in $\mathbb{P}^2(\mathbb{R})$ is a selfdual curve.

We find that there are only two classes of irreducible affine Tits arrangements satisfying the above property: namely the arrangement of type $\tilde{A}_2$ whose corresponding projective root vectors are contained in the union of three projective lines, and a class of arrangements which we call $\tilde{A}^0_2$ (see Figure \ref{fig_A2tilde}) and which is new to the literature to our knowledge. The projective root vectors of $\tilde{A}^0_2$ are contained in the union of a projective conic $\sigma$ and a projective line $\mathfrak{l}$ touching $\sigma$. It turns out that the arrangement $\tilde{A}^0_2$ is an example of an irreducible affine Tits arrangement which is not locally spherical. More precisely, we have the following main theorem (precise definitions are given in Section 2): \begin{theorema}
Let the pair $(\mathcal{A},T)$ be an affine rank three Tits arrangement and assume that the projective root vectors of $\mathcal{A}$ are contained in the locus of a homogeneous cubic polynomial. Then $\mathcal{A}$ is either a near pencil, an arrangement of type $\tilde{A}_2$, or it is an arrangement of type $\tilde{A}^0_2$.
\end{theorema}  
This result is established by proving Theorem 2 in Section 3. The necessary definitions and notations are collected in Section 2. In Section 4 we discuss some related open questions.   
\begin{theoremb}
We wish to thank Bernhard Mühlherr for many helpful discussions.
The second author was supported by a grant of the Deutsche Forschungsgemeinschaft (DFG).
\end{theoremb}

\begin{figure}
\begin{center}
\setlength{\unitlength}{0.8pt}
\begin{picture}(400,400)(100,200)
\Line(102.508119059996927894299376569,431.574625299117189753835264053)(479.133229792436283530961038267,311.054589864736595950103532310)
\Line(102.707793961343826392680997889,432.799168227265161327298901980)(478.630364220274469457622859553,310.048941203940869714256661436)
\Line(102.923188609133078626633322705,434.069493861938513804466623654)(478.103145429856645381373316028,309.009508255030658219553292546)
\Line(103.155681927391490264041209771,435.388055082608465313528950591)(477.549934296165335231592316026,307.934692574089709579894531440)
\Line(103.406794405876541414686380240,436.757468822359148796641888820)(476.968952998137563236652695235,306.822804964181402075957953169)
\Line(103.678204838759305264038958802,438.180528344535120363908323648)(476.358271133873460022772059593,305.672060328494532005247665589)
\Line(103.971769287387754255618203253,439.660216384720659677968081555)(475.715790267735064633731969600,304.480572391867092267744893792)
\Line(104.289542590099188212288643930,441.199719178744415717299011000)(475.039226719203947365713838174,303.246348340007761148582659654)
\Line(104.633802791531732579550760558,442.802441382494594268753631962)(474.326092380417268300615017048,301.967283443871533041681534252)
\Line(105.007078920915471793117538857,444.472021868202851833916415245)(473.573673324097880382501726345,300.641155759643863116107951836)
\Line(105.412182614397689472154553913,446.212350352557130891149242990)(472.779005936165211686473766470,299.265621023850122005421557968)
\Line(105.852244151607916073451173779,448.027584772015523162697762492)(471.938850277822937014142415734,297.838207899722181238311611947)
\Line(106.330753562420518996066012397,449.922169266770724997456954073)(471.049660350728014723283078486,296.356313776957042585997136772)
\Line(106.851607557154303715121846788,451.900852562790021289688103783)(470.107550906603473952036984248,294.817201384649839565616692989)
\Line(107.419163142961051155706930565,453.968706445889072833030206527)(469.108260410428298132548704982,293.217996549236963065544973453)
\Line(108.038298911166102420347837365,456.131143896068845573566338903)(468.047109735771204454474844811,291.555687519106513215108278357)
\Line(108.714485114360032827643457652,458.393936285676191231708260955)(466.918956144399197248739005688,289.827126389187172680604473643)
\Line(109.453863796427602425991477880,460.763228830351308534075802553)(465.718142083588690699727246410,288.029033297182296037719066296)
\Line(110.263340389897240080066387334,463.245553203367586639977870889)(464.438438328755725263942464310,286.158004233938344048039832400)
\Line(111.150688346628303412052290395,465.847835864572566209274398545)(463.072981013379648129912357131,284.210523520442839903282106036)
\Line(112.124668509344871870062686003,468.577400193332275015977632021)(461.614202132855438175621243039,282.182982260793306319679734935)
\Line(113.195165046150413552054197611,471.441959924578047478877714150)(460.053753197821728470781051165,280.071704392621459937511174257)
\Line(114.373339832352282733246561532,474.449600636972045118155886538)(458.382421864658603021839894085,277.872982332797004953245410793)
\Line(115.671807135579911989538561208,477.608745095750332908432657914)(456.590041611775623616075205062,275.583124665412133426040572667)
\Line(117.104830285648843891944832053,480.928097068686148900352706369)(454.665394893963907135704530297,273.198518848184571519577507450)
\Line(118.688541610126458321598742365,484.416556767835273354187869442)(452.596110737947282136981111942,270.715712526029946112343152913)
\Line(120.441186177731945682113147049,488.083099279828242891935633190)(450.368558497636860484050912541,268.131517733224966357310456195)
\Line(122.383388659204087107267942124,491.936605200609015648868148296)(447.967740539726776735392113307,265.443143022854101124955681386)
\Line(124.538440678746942829015524366,495.985630177409339167999774108)(445.377188067759286228724731643,262.648359349036725786393077906)
\Line(126.932603110842534968147222469,500.238097218626656228089371247)(442.578866220102215231232168322,259.745706278238328408596079167)
\Line(129.595413521875778594397367418,504.700892580815583804493170621)(439.553097110627837678784158977,256.734745709813936536983737373)
\Line(132.559982928911570982598493032,509.379343037127515326243912529)(436.278512737500509505969946595,253.616370566736830253931109528)
\Line(135.863257750767538681955534051,514.276549842078394518272812066)(432.732053762043038986469247301,250.393175609833505358705066955)
\Line(139.546211747870398464437277103,519.392553517967596360540686105)(428.889035093751650038179271738,247.069896251255830173489397038)
\Line(143.653918473606143194739788983,524.723304924710063485271275573)(424.723304924710063485271275573,243.653918473606143194739788983)
\Line(148.235437181601780264601486476,530.259423737940974397769478981)(420.207530035013145320440740107,240.155858027635518338207608441)
\Line(153.343424790982632975966281889,535.984737922907483278906742397)(415.313646187879447981212826343,236.590199183596837558624977306)
\Line(159.033365118135361367474024865,541.874620175995761920043033952)(410.013517078804737406372099500,232.975971609018068333706942095)
\Line(165.362287729641538397199668829,547.894173092127488817078172784)(404.279846732742020015610953145,229.337427754660179992529793696)
\Line(172.386838727387210577398105474,553.996367067571173112886547503)(398.087385771836409452744851353,225.704662275645065657836735315)
\Line(180.160574128663656537038434569,560.120304792454356220749569245)(391.414458139485742725892859414,222.114090375139018318582489494)
\Line(188.730386073329894086941470335,566.189870379063061636559561739)(384.244807810097605983809825191,218.608676180364907153238041994)
\Line(198.132057266544635946015405264,572.113108865226934787980254029)(376.569721452108329397993446381,215.237780494099547884024171794)
\Line(208.385081486462331185530199966,577.782751429259449684513464634)(368.390321927383014048313991708,212.056487564297888855293369224)
\Line(219.487088886944545587824413948,583.078319699797418445812627360)(359.719852525646932777114368225,209.124283329924385941039415955)
\Line(231.408454883362794177792604018,587.870167771555265311074731698)(350.585693670838604643566009252,206.503003651632671820599834949)
\Line(244.087901800891660685601089265,592.025616194749563937606707496)(341.030791563021848340286470526,204.254057146228813318865182452)
\Line(257.430044294428086414769689859,595.416987161366412467906051975)(331.114157389490237514592808999,202.435050654368273268849416560)
\Line(271.305790064385691470665047327,597.930902883230674914690685941)(320.910141849045990173645368564,201.096088605948285290848116045)
\Line(285.556216402742437441011621557,599.477760954437207623825119637)(310.506317249966195297355514885,200.276147398857081922322826391)
\Line(300.000000000000000000000000000,600.000000000000000000000000000)(300.000000000000000000000000000,200.000000000000000000000000000)
\Line(314.443783597257562558988378443,599.477760954437207623825119637)(289.493682750033804702644485115,200.276147398857081922322826391)
\Line(328.694209935614308529334952673,597.930902883230674914690685941)(279.089858150954009826354631436,201.096088605948285290848116045)
\Line(342.569955705571913585230310141,595.416987161366412467906051975)(268.885842610509762485407191001,202.435050654368273268849416560)
\Line(355.912098199108339314398910735,592.025616194749563937606707496)(258.969208436978151659713529474,204.254057146228813318865182452)
\Line(368.591545116637205822207395982,587.870167771555265311074731698)(249.414306329161395356433990748,206.503003651632671820599834949)
\Line(380.512911113055454412175586051,583.078319699797418445812627360)(240.280147474353067222885631775,209.124283329924385941039415955)
\Line(391.614918513537668814469800034,577.782751429259449684513464634)(231.609678072616985951686008293,212.056487564297888855293369224)
\Line(401.867942733455364053984594737,572.113108865226934787980254029)(223.430278547891670602006553619,215.237780494099547884024171794)
\Line(411.269613926670105913058529666,566.189870379063061636559561739)(215.755192189902394016190174809,218.608676180364907153238041994)
\Line(419.839425871336343462961565431,560.120304792454356220749569245)(208.585541860514257274107140586,222.114090375139018318582489494)
\Line(427.613161272612789422601894526,553.996367067571173112886547503)(201.912614228163590547255148647,225.704662275645065657836735315)
\Line(434.637712270358461602800331171,547.894173092127488817078172784)(195.720153267257979984389046855,229.337427754660179992529793696)
\Line(440.966634881864638632525975135,541.874620175995761920043033952)(189.986482921195262593627900500,232.975971609018068333706942095)
\Line(446.656575209017367024033718111,535.984737922907483278906742397)(184.686353812120552018787173657,236.590199183596837558624977306)
\Line(451.764562818398219735398513523,530.259423737940974397769478981)(179.792469964986854679559259893,240.155858027635518338207608441)
\Line(456.346081526393856805260211017,524.723304924710063485271275573)(175.276695075289936514728724427,243.653918473606143194739788983)
\Line(460.453788252129601535562722897,519.392553517967596360540686105)(171.110964906248349961820728262,247.069896251255830173489397038)
\Line(464.136742249232461318044465949,514.276549842078394518272812066)(167.267946237956961013530752699,250.393175609833505358705066955)
\Line(467.440017071088429017401506968,509.379343037127515326243912529)(163.721487262499490494030053405,253.616370566736830253931109528)
\Line(470.404586478124221405602632582,504.700892580815583804493170621)(160.446902889372162321215841022,256.734745709813936536983737373)
\Line(473.067396889157465031852777531,500.238097218626656228089371247)(157.421133779897784768767831678,259.745706278238328408596079167)
\Line(475.461559321253057170984475634,495.985630177409339167999774108)(154.622811932240713771275268356,262.648359349036725786393077906)
\Line(477.616611340795912892732057876,491.936605200609015648868148296)(152.032259460273223264607886693,265.443143022854101124955681386)
\Line(479.558813822268054317886852952,488.083099279828242891935633190)(149.631441502363139515949087458,268.131517733224966357310456195)
\Line(481.311458389873541678401257635,484.416556767835273354187869442)(147.403889262052717863018888058,270.715712526029946112343152913)
\Line(482.895169714351156108055167947,480.928097068686148900352706369)(145.334605106036092864295469703,273.198518848184571519577507450)
\Line(484.328192864420088010461438792,477.608745095750332908432657914)(143.409958388224376383924794938,275.583124665412133426040572667)
\Line(485.626660167647717266753438468,474.449600636972045118155886538)(141.617578135341396978160105915,277.872982332797004953245410793)
\Line(486.804834953849586447945802389,471.441959924578047478877714150)(139.946246802178271529218948834,280.071704392621459937511174257)
\Line(487.875331490655128129937313997,468.577400193332275015977632021)(138.385797867144561824378756961,282.182982260793306319679734935)
\Line(488.849311653371696587947709605,465.847835864572566209274398545)(136.927018986620351870087642869,284.210523520442839903282106036)
\Line(489.736659610102759919933612666,463.245553203367586639977870889)(135.561561671244274736057535690,286.158004233938344048039832400)
\Line(490.546136203572397574008522120,460.763228830351308534075802553)(134.281857916411309300272753590,288.029033297182296037719066296)
\Line(491.285514885639967172356542348,458.393936285676191231708260955)(133.081043855600802751260994312,289.827126389187172680604473643)
\Line(491.961701088833897579652162635,456.131143896068845573566338903)(131.952890264228795545525155190,291.555687519106513215108278357)
\Line(492.580836857038948844293069435,453.968706445889072833030206527)(130.891739589571701867451295018,293.217996549236963065544973453)
\Line(493.148392442845696284878153212,451.900852562790021289688103783)(129.892449093396526047963015752,294.817201384649839565616692989)
\Line(493.669246437579481003933987603,449.922169266770724997456954073)(128.950339649271985276716921515,296.356313776957042585997136772)
\Line(494.147755848392083926548826220,448.027584772015523162697762492)(128.061149722177062985857584266,297.838207899722181238311611947)
\Line(494.587817385602310527845446087,446.212350352557130891149242990)(127.220994063834788313526233530,299.265621023850122005421557968)
\Line(494.992921079084528206882461143,444.472021868202851833916415245)(126.426326675902119617498273655,300.641155759643863116107951836)
\Line(495.366197208468267420449239442,442.802441382494594268753631962)(125.673907619582731699384982952,301.967283443871533041681534252)
\Line(495.710457409900811787711356070,441.199719178744415717299011000)(124.960773280796052634286161826,303.246348340007761148582659654)
\Line(496.028230712612245744381796747,439.660216384720659677968081555)(124.284209732264935366268030400,304.480572391867092267744893792)
\Line(496.321795161240694735961041198,438.180528344535120363908323648)(123.641728866126539977227940407,305.672060328494532005247665589)
\Line(496.593205594123458585313619760,436.757468822359148796641888820)(123.031047001862436763347304765,306.822804964181402075957953169)
\Line(496.844318072608509735958790229,435.388055082608465313528950591)(122.450065703834664768407683974,307.934692574089709579894531440)
\Line(497.076811390866921373366677296,434.069493861938513804466623654)(121.896854570143354618626683971,309.009508255030658219553292546)
\Line(497.292206038656173607319002111,432.799168227265161327298901980)(121.369635779725530542377140447,310.048941203940869714256661436)
\Line(497.491880940003072105700623431,431.574625299117189753835264053)(120.866770207563716469038961733,311.054589864736595950103532310)
\Line(102.953870994032651567224415107,434.246504110108424055419412257)(476.503114469928398995366808567,305.943365027153176280458160836)
\Line(103.196008458029964381407702594,435.611640135609967184552318134)(475.847095041850570347932876648,304.722514908597778633955900158)
\Line(103.458492876040393070365945430,437.032903983383953337848994451)(475.154254105566762496136672933,303.454739791522914325462639255)
\Line(103.743296014546234738379731027,438.513713671444047989877191242)(474.421656252700783299470727910,302.137413532725971871335792286)
\Line(104.052619451026039636842781344,440.057758998670672713578568944)(473.646063766075701157450208305,300.767724310343605416732943365)
\Line(104.388926337811351338246946551,441.669027595155209899734544082)(472.823898484013609439450992084,299.342659915993416764625962764)
\Line(104.754978274557832307434319313,443.351833772420930232160565549)(471.951197976386475013164600239,297.858992003771845386493850465)
\Line(105.153878228748875798853882277,445.110850476385153940391268827)(471.023565055255699494161311316,296.313259305778535004172004444)
\Line(105.589120638079591068179714438,446.951144668951674154819419009)(470.036109453779971916720478686,294.701749863437154740285418279)
\Line(106.064650067628734702175405615,448.878216483508322303468712980)(468.983380277485279660803525852,293.020482380995945115454984811)
\Line(106.584930089609509915410269027,450.898042512053047106226160991)(467.859287552649650777123559613,291.265186888849263934771451197)
\Line(107.155024418406832246016751861,453.017123582242132710295076209)(466.657010855876428181503273453,289.431285018843208195454168613)
\Line(107.780692786637294030251308129,455.242537364017666853966001674)(465.368892592537449550066474666,287.513870354084043417961493591)
\Line(108.468504612153082764035500766,457.581996096828521379390092934)(463.986312981839520441578135017,285.507689539331196315065222424)
\Line(109.225974210667854278346221510,460.043909633960656797309843748)(462.499543180767812870840142093,283.407125148910677501062883457)
\Line(110.061722191586187572369380539,462.637453836932237450657865083)(460.897572210961886735906155507,281.206181740722277452276457460)
\Line(110.985668783658972672006544477,465.372644086340349406971384107)(459.167902408526390971724461776,278.898477124068259271557604131)
\Line(112.009266233469567429429376787,468.260413256304525840195213914)(457.296306956779650555006101678,276.477241701140134764846887435)
\Line(113.145779181289019223618982530,471.312692855005804309463353797)(455.266541638338712199255554762,273.935329899015596823519177510)
\Line(114.410624143877663608145427529,474.542495057081068957969587052)(453.060001200553589440105169322,271.265249320602893555603066753)
\Line(115.821782045509826574276730331,477.963991887975525324934864940)(450.655308606010128399908375762,268.459215492578688428450259977)
\Line(117.400301288736946621606645147,481.592585634700644674727169790)(448.027822875098165245118006102,265.509243236352616769497354460)
\Line(119.170913327921638646221680272,485.444961310436392170493296908)(445.149048166782307758771935963,262.407290104900144882958911435)
\Line(121.162788312303247760916917641,489.539107186580435316769860259)(441.985923162420316267174717592,259.145473541972127592347741803)
\Line(123.410465326973012163383764443,493.894282275140690507377306124)(438.499965760176214805540196052,255.716392183900474336928349948)
\Line(125.955000284347228794855236311,498.530899082361096611756594492)(434.646243740034983510570956131,252.113594111226403042595238974)
\Line(128.845384709540306314757306605,503.470274305110164930604440369)(430.372137880012318851877652512,248.332252391633846629964147678)
\Line(132.140300359596646939635896635,508.734176948343019235838825933)(425.615860957565991431052269721,244.370133085292885860192977790)
\Line(135.910287069971932922585264621,514.344068978416977522824713906)(420.304696038540235255020543778,240.228975996664396961441312309)
\Line(140.240412477515593941519077690,520.319882790192431583832281775)(414.352923953192637015029163256,235.916457914404646278893489925)
\Line(145.233538013026701661193270513,526.678104832819270651174167506)(407.659429256189164841443901515,231.448977184261499573392207630)
\Line(151.014266998356265536691678038,533.428825079002044820559842097)(400.105017556992578228341341886,226.855593622218890738529486592)
\Line(157.733620766721146918706243343,540.571253604187806913269670765)(391.549564483679088124635847821,222.183585564074865696684677433)
\Line(165.574385532572282478702469486,548.086981787915739547226501412)(381.829280905799698625329356204,217.506250007185931424994317894)
\Line(174.756848528300077291344440922,555.929961872107278100169357498)(370.754674865721694218828470313,212.933765781619448477072306065)
\Line(185.544196306974701327274645541,564.011795310531420333162450601)(358.110303081731837776149539511,208.628129873407812596865639256)
\Line(198.246016103727382507772924763,572.180506333432220016182310723)(343.658256225787119520495057251,204.823268130333937036673765490)
\Line(213.216876043407241756302640532,580.190702857652053961749772660)(327.148660649016086939771503170,201.851191714498152975136191069)
\Line(230.844557729900421006921187478,587.663328342078928365449385387)(308.342410343784591839253936668,200.174065272658074203785639442)
\Line(251.518980419472741679799800317,594.035024519885410768491450513)(287.053719305895568335119563925,200.419455316432561663410992092)
\Line(275.568627708308487723411343748,598.502161318069200068819998714)(263.222017386578686962580804621,203.410631022715575722242746649)
\Line(303.148628685367571766672045666,599.975213807615868934003319895)(237.021715095592405164684029669,210.174460015256992122284488755)
\Line(334.069493861938513804466623654,597.076811390866921373366677296)(209.009508255030658219553292546,221.896854570143354618626683971)
\Line(367.578211803859649825457750411,588.237045475625451586876268504)(180.285800587593909496660139995,239.786047863968260277485365403)
\Line(402.152671498622150678660192439,571.944269185380498764488680959)(152.608156327674916760037989326,264.812558205753133941569046358)
\Line(435.433514561181186053002716064,547.165767533099638468192659605)(128.357449620399894854238993810,297.341172327040282916656072009)
\Line(464.438438328755725263942464310,513.841995766061655951960167599)(110.263340389897240080066387334,336.754446796632413360022129111)
\Line(486.109246349690127260771268720,473.234885288026170577487542711)(100.881763149469669459857097563,381.240262434242180654881173480)
\Line(498.044069053322654239816515654,427.902449942662170576460825392)(101.955930946677345760183484346,427.902449942662170576460825392)
\Line(499.118236850530330540142902438,381.240262434242180654881173480)(113.890753650309872739228731279,473.234885288026170577487542711)
\Line(489.736659610102759919933612666,336.754446796632413360022129111)(135.561561671244274736057535689,513.841995766061655951960167600)
\Line(471.642550379600105145761006190,297.341172327040282916656072009)(164.566485438818813946997283936,547.165767533099638468192659606)
\Line(447.391843672325083239962010675,264.812558205753133941569046358)(197.847328501377849321339807561,571.944269185380498764488680959)
\Line(419.714199412406090503339860005,239.786047863968260277485365403)(232.421788196140350174542249589,588.237045475625451586876268504)
\Line(390.990491744969341780446707454,221.896854570143354618626683971)(265.930506138061486195533376346,597.076811390866921373366677296)
\Line(362.978284904407594835315970331,210.174460015256992122284488755)(296.851371314632428233327954334,599.975213807615868934003319895)
\Line(336.777982613421313037419195379,203.410631022715575722242746649)(324.431372291691512276588656252,598.502161318069200068819998714)
\Line(312.946280694104431664880436075,200.419455316432561663410992092)(348.481019580527258320200199683,594.035024519885410768491450513)
\Line(291.657589656215408160746063333,200.174065272658074203785639442)(369.155442270099578993078812522,587.663328342078928365449385387)
\Line(272.851339350983913060228496830,201.851191714498152975136191069)(386.783123956592758243697359469,580.190702857652053961749772660)
\Line(256.341743774212880479504942750,204.823268130333937036673765490)(401.753983896272617492227075237,572.180506333432220016182310723)
\Line(241.889696918268162223850460489,208.628129873407812596865639256)(414.455803693025298672725354459,564.011795310531420333162450601)
\Line(229.245325134278305781171529687,212.933765781619448477072306064)(425.243151471699922708655559078,555.929961872107278100169357498)
\Line(218.170719094200301374670643796,217.506250007185931424994317894)(434.425614467427717521297530514,548.086981787915739547226501412)
\Line(208.450435516320911875364152179,222.183585564074865696684677433)(442.266379233278853081293756657,540.571253604187806913269670765)
\Line(199.894982443007421771658658113,226.855593622218890738529486592)(448.985733001643734463308321961,533.428825079002044820559842097)
\Line(192.340570743810835158556098485,231.448977184261499573392207630)(454.766461986973298338806729487,526.678104832819270651174167506)
\Line(185.647076046807362984970836744,235.916457914404646278893489925)(459.759587522484406058480922310,520.319882790192431583832281775)
\Line(179.695303961459764744979456222,240.228975996664396961441312309)(464.089712930028067077414735379,514.344068978416977522824713906)
\Line(174.384139042434008568947730278,244.370133085292885860192977791)(467.859699640403353060364103365,508.734176948343019235838825933)
\Line(169.627862119987681148122347488,248.332252391633846629964147678)(471.154615290459693685242693395,503.470274305110164930604440369)
\Line(165.353756259965016489429043869,252.113594111226403042595238974)(474.044999715652771205144763689,498.530899082361096611756594492)
\Line(161.500034239823785194459803948,255.716392183900474336928349948)(476.589534673026987836616235557,493.894282275140690507377306124)
\Line(158.014076837579683732825282408,259.145473541972127592347741803)(478.837211687696752239083082358,489.539107186580435316769860259)
\Line(154.850951833217692241228064037,262.407290104900144882958911435)(480.829086672078361353778319729,485.444961310436392170493296908)
\Line(151.972177124901834754881993898,265.509243236352616769497354460)(482.599698711263053378393354852,481.592585634700644674727169790)
\Line(149.344691393989871600091624238,268.459215492578688428450259977)(484.178217954490173425723269669,477.963991887975525324934864940)
\Line(146.939998799446410559894830678,271.265249320602893555603066752)(485.589375856122336391854572471,474.542495057081068957969587052)
\Line(144.733458361661287800744445238,273.935329899015596823519177510)(486.854220818710980776381017470,471.312692855005804309463353797)
\Line(142.703693043220349444993898322,276.477241701140134764846887435)(487.990733766530432570570623213,468.260413256304525840195213914)
\Line(140.832097591473609028275538224,278.898477124068259271557604131)(489.014331216341027327993455523,465.372644086340349406971384107)
\Line(139.102427789038113264093844493,281.206181740722277452276457460)(489.938277808413812427630619461,462.637453836932237450657865083)
\Line(137.500456819232187129159857907,283.407125148910677501062883457)(490.774025789332145721653778490,460.043909633960656797309843748)
\Line(136.013687018160479558421864983,285.507689539331196315065222424)(491.531495387846917235964499234,457.581996096828521379390092934)
\Line(134.631107407462550449933525334,287.513870354084043417961493591)(492.219307213362705969748691871,455.242537364017666853966001674)
\Line(133.342989144123571818496726546,289.431285018843208195454168613)(492.844975581593167753983248138,453.017123582242132710295076209)
\Line(132.140712447350349222876440386,291.265186888849263934771451197)(493.415069910390490084589730973,450.898042512053047106226160991)
\Line(131.016619722514720339196474148,293.020482380995945115454984811)(493.935349932371265297824594385,448.878216483508322303468712980)
\Line(129.963890546220028083279521315,294.701749863437154740285418279)(494.410879361920408931820285562,446.951144668951674154819419009)
\Line(128.976434944744300505838688684,296.313259305778535004172004444)(494.846121771251124201146117723,445.110850476385153940391268827)
\Line(128.048802023613524986835399761,297.858992003771845386493850465)(495.245021725442167692565680687,443.351833772420930232160565549)
\Line(127.176101515986390560549007915,299.342659915993416764625962764)(495.611073662188648661753053449,441.669027595155209899734544082)
\Line(126.353936233924298842549791694,300.767724310343605416732943365)(495.947380548973960363157218656,440.057758998670672713578568944)
\Line(125.578343747299216700529272090,302.137413532725971871335792286)(496.256703985453765261620268973,438.513713671444047989877191242)
\Line(124.845745894433237503863327067,303.454739791522914325462639255)(496.541507123959606929634054570,437.032903983383953337848994451)
\Line(124.152904958149429652067123353,304.722514908597778633955900158)(496.803991541970035618592297406,435.611640135609967184552318135)
\Line(123.496885530071601004633191433,305.943365027153176280458160836)(497.046129005967348432775584892,434.246504110108424055419412257)
\Line(122.875021242256737664319149538,307.119744293692870539362813163)(497.269688907458089462058395088,432.934326140892981902169314882)
\Line(122.284885674639522987029818917,308.253947549095850119429350492)(497.476262022085234594591137461,431.672163452860886832442942118)
\Line(121.724266854208738204932083002,309.348122074991613245453621205)(497.667281136140734546969914173,430.457281038298565087183929233)
\Line(121.191144851280514890144739172,310.404278448107505179116466397)(497.844038996971726244388633446,429.287134263439480841760973239)
\thinlines
\strokepath
\end{picture}
\end{center}
A subset of the arrangement of type $\tilde{A}^0_2$\caption{\label{fig_A2tilde}}
\end{figure}

\section{Definitions and notation}
We start with the notion of a Tits arrangement in $\mathbb{R}^r$ (see \cite{CMW}).

\begin{definition}
Let $\mathcal{A}$ be a (possibly infinite) set of linear hyperplanes in $V:=\mathbb{R}^r$ and let $T$ be an open convex cone in $V$. We say that $\mathcal{A}$ \textit{is locally finite in} $T$ if for every $x \in T$  there exists a neighbourhood $U_x \subset T$ of $x$, such that $\lbrace H \in \mathcal{A} \mid H \cap U_x \neq \emptyset \rbrace$ is finite. A \textit{hyperplane arrangement (of rank r)} is a pair $(\mathcal{A},T)$, where $T$ is a convex open cone in $V$, and $\mathcal{A}$ is a set of linear hyperplanes such that the following holds: \begin{itemize}
\item $H \cap T \neq \emptyset$ for all $H \in \mathcal{A}$,
\item $\mathcal{A}$ is locally finite in $T$.
\end{itemize}
Denote by $\overline{T}$ the topological closure of $T$ with respect to the standard topology of $V$. If $X \subset \overline{T}$ then the \textit{localization at} $X$ \textit{(in} $\mathcal{A}$\textit{)} is defined as \begin{align*}
\mathcal{A}_X:= \lbrace H \in \mathcal{A} \: | \: X \subset H \rbrace.
\end{align*} 
In the case $X= \lbrace x \rbrace$ we write $\mathcal{A}_x$ instead of $A_{ \lbrace x \rbrace }$ and call $(\mathcal{A}_x, T)$ the \textit{parabolic subarrangement at} $x$.  The connected components of $T \setminus \bigcup_{H \in \mathcal{A}} H$ are called \textit{chambers} or \textit{cells}. If $K$ is a chamber then its \textit{walls} are given by the hyperplanes contained in the set \begin{align*}
W^K:= \lbrace H \leq V \: | \: \text{dim}(H)= r-1, \langle H \cap \overline{K} \rangle_{\mathbb{R}}=H, H \cap K = \emptyset \rbrace.
\end{align*}
The arrangement $(\mathcal{A}, T)$ is called \textit{thin} if $W^K \subset \mathcal{A}$ for each chamber $K$. A \textit{simplicial hyperplane arrangement (of rank $r$)} is an arrangement $(\mathcal{A}, T)$ such that each chamber $K$ is an open simplicial cone. $T$ is called the \textit{Tits cone} of the arrangement. Finally, a simplicial arrangement is called a \textit{Tits arrangement} if it is also thin. 
\end{definition}

\begin{remark}
If the pair $(\mathcal{A},T)$ is a Tits arrangement, we usually omit the reference to $T$, since it should always be clear from the context.
\end{remark}

\begin{definition}
Let the pair $(\mathcal{A}, T)$ be a Tits arrangement and denote the set of chambers by $\mathcal{K}$. Then we have the following thin chamber complex \begin{align*}
\mathcal{S}(\mathcal{A}, T):= \left\lbrace \overline{K} \cap \bigcap_{H \in X} H \: | \: K \in \mathcal{K}, X \subset W^K \right\rbrace,
\end{align*} 
whose poset-structure is given by set-wise inclusion. The $0$-simplices of $\mathcal{S}(\mathcal{A}, T)$ are called \textit{vertices} and the $1$-simplices of $\mathcal{S}(\mathcal{A}, T)$ are called \textit{segments} or \textit{edges}. We call the Tits arrangement $(\mathcal{A}, T)$ \textit{locally spherical} if all vertices meet $T$. If $v$ is a vertex of $\mathcal{A}$ then we define its \textit{weight} to be $w(v):=|\mathcal{A}_{v}|$. If $\overline{K}$ is the closure of a chamber $K$ with supporting hyperplanes $H_1, H_2, H_3 \in \mathcal{A}$, then the \textit{Coxeter diagram} $\Gamma^K$ associated to $K$ is the weighted undirected graph with vertices $H_1, H_2, H_3$ and edges between $H_i, H_j$ if and only if $|\mathcal{A}_{H_i \cap H_j}|>2$; in this case the weight of the edge between $H_i, H_j$ is exactly the quantity $|\mathcal{A}_{H_i \cap H_j}|$. We agree that edge weights less than four are omitted. The Tits arrangement $\mathcal{A}$ is called \textit{irreducible} if $\Gamma^K$ is connected for any chamber $K$ of $\mathcal{A}$. A rank three Tits arrangement which is not irreducible is called a \textit{near pencil}. 
\end{definition}

\begin{definition}
Let the pair $(\mathcal{A}, T)$ be a Tits arrangement in $V:=\mathbb{R}^r$ and denote by $\partial T$ the boundary of $T$. If there is a linear form $0 \neq \alpha \in V^*$ such that $\partial T=\text{ker}(\alpha)$, then we say that $(\mathcal{A}, T)$ is an \textit{affine Tits arrangement}.
\end{definition}

In the following sections we will be concerned with the case of an affine Tits arrangement $(\mathcal{A},T)$ of rank three. Then we may view $\mathcal{A}$ as set of projective lines in $\mathbb{P}^2(\mathbb{R})$ and the boundary $\partial T$ of $T$ is again a projective line.
Further, in $\mathbb{P}^2(\mathbb{R})$ we have a duality between projective lines and projective points, for which we require the following notation.

\begin{notatio}
Let $(\mathcal{A},T)$ be a Tits arrangement of rank three. By abuse of notation we denote the set of projective lines $\lbrace  \mathfrak{g} \mid \exists H \in \mathcal{A}: \mathfrak{g}=\pi(H)  \rbrace$ by $\mathcal{A}$ as well; here $\pi: \{U\le \mathbb{R}^3 \mid \dim U \ge 1\} \longrightarrow \{ U\le \mathbb{P}^2(\mathbb{R}) \}$ is the natural projection. If $p \in  (\mathbb{P}^2(\mathbb{R}))^*$ then we denote the corresponding dual line by $p^* \subset \mathbb{P}^2(\mathbb{R})$. Likewise, if $\mathfrak{l} \subset (\mathbb{P}^2(\mathbb{R}))^*$ is a projective line, then its corresponding dual point is denoted by $\mathfrak{l}^* \in \mathbb{P}^2(\mathbb{R})$. Similarly, if $\mathcal{A}$ is a set of projective lines in $\mathbb{P}^2(\mathbb{R})$, we write $\mathcal{A}^* \subset (\mathbb{P}^2(\mathbb{R}))^*$ for the corresponding set of dual projective points (and vice versa).    
\end{notatio}

\section{Results and proofs}

Now we are ready to prove our main theorem.
The main strategy can be summarized as follows: according to the possible factorizations of a homogeneous cubic polynomial $P$, there are naturally three cases to consider. Namely, $P$ may factor as a product of three linear polynomials, or it may factor as a product of an irreducible quadratic polynomial and a linear polynomial, or $P$ may be irreducible. We examine all three cases and collect all (up to projectivity) affine Tits arrangements $\mathcal{A}$ such that $\mathcal{A}^* \subset V(P)$. 

We start with the following lemma which will be used extensively to rule out the possibility of existence of certain Tits arrangements. It basically says that near pencils are the only rank three Tits arrangements containing a segment bounded by two vertices of weight two.
\begin{lemma} \label{near pencil lemma}
Let $\mathcal{A}$ be an Tits arrangement of rank three. Suppose there is a line $\mathfrak{g} \in \mathcal{A}$ containing two vertices $v_1, v_2$ of weight two such that there is no other vertex contained in the bounded segment between $v_1$ and $v_2$ on $\mathfrak{g}$. Then $\mathcal{A}$ is a near pencil. 
\end{lemma}

\begin{proof}
Denote by $\mathfrak{g}_1, \mathfrak{g}_2$ the two lines meeting $\mathfrak{g}$ in $v_1$ respectively $v_2$ and set $v:=\mathfrak{g}_1 \cap \mathfrak{g}_2$. Using $w(v_1)=w(v_2)=2$ it follows that there are two chambers with vertices $v_1, v_2, v$ and it is easy to see that every line $\mathfrak{g}^\prime \in \mathcal{A} \setminus \lbrace \mathfrak{g} \rbrace$ needs to pass through $v$.  
\end{proof}
We state two further lemmas, which will turn out to be useful and may be interesting in their own right. 
\begin{lemma} \label{nur ein dicker vertex lemma}
Let $\mathcal{A}$ be an affine Tits arrangement of rank three. Then there is at most one vertex of $\mathcal{A}$ contained in $\partial T$.
\end{lemma}

\begin{proof}
Suppose there were two vertices $v \neq w \in \partial T$. Then there is a chamber $K$ having $v$ as a vertex. As $\mathcal{A}$ is thin, it follows that $K$ has to be contained in the cone $C$ generated by two neighbouring lines passing through $v$. As the lines passing through $w$ accumulate at $\partial T$ we conclude that there are infinitely many lines passing through $w$ and intersecting $K$, a contradiction.
\end{proof}

\begin{lemma} \label{4 mal 3  lemma}
Let $\mathcal{A}$ be a Tits arrangement of rank three. Suppose there is a vertex $v$ of weight two which is surrounded by neighbouring vertices $v_1, v_2, v_3, v_4$ of weight three. Then $\mathcal{A}$ is spherical and $|\mathcal{A}| \in \lbrace 6,7 \rbrace $.
\end{lemma}

\begin{proof}
We denote the lines intersecting in $v$ by $\mathfrak{g}_1, \mathfrak{g}_2$ and we agree that $v_1, v_3 \in \mathfrak{g}_1$ while $v_2, v_4 \in \mathfrak{g}_2$. It is clear that there are no further vertices lying in the bounded segment between $v_1$ and $v_4$ and the same is true for the bounded segments between $v_1$ and $v_2$, $v_2$ and $v_3$, $v_3$ and $v_4$. Denote the line passing through $v_i$ and $v_j$ by $\mathfrak{g}_{i,j}$ and observe that the spherical arrangement $\mathcal{B} \subset \mathcal{A}$ defined by $\mathcal{B}:=\lbrace \mathfrak{g}_{i,j} \mid 1 \leq i,j \leq 4 \rbrace$ is simplicial. Now by the above there are cells $K_{i,j}$ of $\mathcal{A}$ containing the vertices $v_i$ and $v_j$ for $\lbrace i, j \rbrace \in \left\lbrace \lbrace 1,4 \rbrace, \lbrace 1,2 \rbrace, \lbrace 2,3 \rbrace, \lbrace 3,4 \rbrace \right\rbrace$ and these cells are triangles. Suppose there was a line in $\mathcal{A}$ not contained in $\mathcal{B}$ supporting an edge of such a cell $K_{i,j}$. This edge needs to pass through either $v_i$ or $v_j$. But then the weight of either $v_i$ or $v_j$ needs to be strictly greater than three, contradicting our assumption. This shows that the only line one may add to $\mathcal{B}$ in such a way that the obtained arrangement is simplicial with the vertices $v_1, v_2, v_3, v_4$ having weight three is the line passing through the points $\mathfrak{g}_{1,2} \cap \mathfrak{g}_{3,4}$ and $\mathfrak{g}_{1,4} \cap \mathfrak{g}_{2,3}$.
\end{proof}

The next proposition and the following theorem are preliminary results which will be used to simplify the proof of Proposition 2.

\begin{proposition}
Let $\mathcal{A}$ be a Tits arrangement of rank three. Assume that $\mathcal{A}^*$ is contained in the union of two lines. Then $\mathcal{A}$ is a near pencil.
\end{proposition}

\begin{proof}
Suppose $\mathcal{A}^* \subset \mathfrak{l}_1 \cup \mathfrak{l}_2$. Then after dualizing the lines $\mathfrak{l}_1, \mathfrak{l}_2 \subset (\mathbb{P}^2(\mathbb{R}))^*$ become two points $v_1, v_2 \in \mathbb{P}^2(\mathbb{R})$. Suppose that $w(v_1)=|\mathcal{A}_{v_1}| \geq 3$ and pick a line $\mathfrak{g}$ of $\mathcal{A}$ such that $v_1 \notin \mathfrak{g}$ and $v_2 \in \mathfrak{g}$. Observe that there is at most one line in $\mathcal{A}_{v_1}$ meeting $\mathfrak{g}$ in a vertex of weight greater than two. Further, different lines in $\mathcal{A}_{v_1}$ produce different intersections with $\mathfrak{g}$.  Since $\mathfrak{g}$ contains only one vertex with weight possibly bigger than two, we may use Lemma \ref{near pencil lemma} to conclude that $\mathcal{A}$ is a near pencil. If on the other hand $w(v_1)= 2$ then we choose $\mathfrak{g} \in \mathcal{A}$ passing through $v_1$ but not through $v_2$. Then $\mathfrak{g}$ contains a segment bounded by two vertices of weight two. Hence by Lemma \ref{near pencil lemma} it follows that $\mathcal{A}$ is a near pencil.  
\end{proof}

\begin{remark}
We observe that in the situation of the last proposition there is a unique
$p \in \mathcal{A}^*$ such that $\mathcal{A}^* \setminus \lbrace p \rbrace$ is contained in one of the two lines $\mathfrak{l}_1, \mathfrak{l}_2$ while $p$ is contained in the other one. 
\end{remark}

\begin{theorem}
The near pencil is the only Tits arrangement $\mathcal{A}$ of rank three such that $\mathcal{A}^*$ lies on a conic.
\end{theorem}

\begin{proof}
Let $P \in \mathbb{R}[x,y,z]$ be a homogeneous polynomial of degree two and set $\sigma:=V(P)$.
Suppose that $\mathcal{A}^* \subset \sigma \subset (\mathbb{P}^2(\mathbb{R}))^*$ for some rank three Tits arrangement $\mathcal{A}$. First, assume that $P$ is the product of two distinct linear polynomials. Then by Proposition 1 the only Tits arrangements lying on $\sigma$ are near pencils. If $P$ factors as a square of a linear polynomial, then every $p \in \mathcal{A}^*$ lies on a single line which means that all lines of $\mathcal{A}$ pass through a single point. Hence $\mathcal{A}$ is not simplicial. Now, finally suppose that $P$ is irreducible. By Bézouts theorem every line meets $\sigma$ in at most two points. Hence the weight of any vertex of $\mathcal{A}$ is bounded by two. But this implies that $\mathcal{A}$ is a near pencil consisting of three lines. 
\end{proof}

The next proposition is a first step towards our main theorem. 

\begin{proposition}
Let $\mathcal{A}$ be an affine Tits arrangement of rank three. Suppose that $\mathcal{A}^*$ is contained in the union of at most three lines. Then $\mathcal{A}$ is either a near pencil or it is an arrangement of type $\tilde{A_2}$.
\end{proposition}

\begin{proof}
Taking into account Theorem 1 it is enough to consider the case where $\mathcal{A}^*$ is contained in the union of exactly three lines: $\mathcal{A}^* \subset \mathfrak{l}_1 \cup \mathfrak{l}_2 \cup \mathfrak{l}_3 \subset (\mathbb{P}^2(\mathbb{R}))^*$.  We define $v_1:=\mathfrak{l}_1^*, v_2:=\mathfrak{l}_2^*, v_3:=\mathfrak{l}_3^* \in \mathbb{P}^2(\mathbb{R})$ and consider two cases: \\
a) Suppose that $\mathfrak{l}_1 \cap \mathfrak{l}_2 \cap \mathfrak{l}_3 =:w$. Then the corresponding points $v_1, v_2, v_3$ all lie on the line $w^* \subset \mathbb{P}^2(\mathbb{R})$. If $| \mathcal{A}_{v_i} |=| \mathcal{A}_{v_j} |= \infty$ for two different values $i,j$, then we have $w^* = \partial T$ because $\mathcal{A}$ is locally finite in $T$.\\
Let $k \neq i,j$ and assume that $| \mathcal{A}_{v_k} |<\infty$. Then it is easy to see that $\mathcal{A}$ contains a segment bounded by two vertices of weight two. By Lemma \ref{near pencil lemma} we may assume that $| \mathcal{A}_{v_k} |=\infty$. 
Observe that all vertices in $T$ have weight bounded by three. But since $\mathcal{A}$ is not spherical, Lemma \ref{4 mal 3  lemma} shows that every vertex in $T$ has weight exactly three. From this it is easy to see that $\mathcal{A}$ is of type $\tilde{A}_2$. \\
Now suppose that there is precisely one $i$ such that $| \mathcal{A}_{v_i} |= \infty$ and pick a line $\mathfrak{g} \in \mathcal{A}$ such that $v_j \in \mathfrak{g}, v_i \notin \mathfrak{g}$ for some $j \neq i$. Then it is easy to see that $\mathfrak{g}$ contains a segment bounded by two vertices of weight two. Hence by Lemma \ref{near pencil lemma} the arrangement $\mathcal{A}$ is a near pencil.  \\
b) Assume $\mathfrak{l}_1 \cap \mathfrak{l}_2 \cap \mathfrak{l}_3=\emptyset$. Then the three points $v_1, v_2, v_3$ are not collinear. Hence it is impossible to have $| \mathcal{A}_{v_i} |= \infty$ for all $1 \leq i \leq 3$. But then we may assume that $| \mathcal{A}_{v_1} |= \infty$ and $| \mathcal{A}_{v_3} |< \infty$. Now if $| \mathcal{A}_{v_2} |< \infty$ as well, then we may argue as in case a) to show that $\mathcal{A}$ is a near pencil. So assume that $| \mathcal{A}_{v_1} |=| \mathcal{A}_{v_2} |= \infty$ and  $| \mathcal{A}_{v_3} |< \infty$. Again, we may argue as in case a) to conclude that $\mathcal{A}$ is a near pencil.  
\end{proof}

\begin{remark}
If we drop the condition on $\mathcal{A}$ to be affine, then we find some more possible (spherical) arrangements such that $A^*$ is contained in the union of three lines: for instance the arrangement of type $A(10,3)$ (as denoted in \cite{p-G-09}) and some of its subarrangements.
\end{remark}

Our next goal is to show that there is no affine Tits arrangement $\mathcal{A}$ such that $\mathcal{A}^*$ is contained in the locus of an irreducible homogeneous cubic polynomial. This may be deduced from Lemma \ref{4 mal 3  lemma} and the following result:

\begin{lemma} \label{a2_affin lemma}
Let $\mathcal{A}$ be an affine Tits arrangement of rank three. Assume that every vertex of $\mathcal{A}$ has weight three and suppose that $\mathcal{A}^* \subset V(F)$ for some homogeneous cubic polynomial $F$. Then $F$ is not irreducible.
\end{lemma}

\begin{proof}
Consider the affine space $\mathbb{E}:=\mathbb{P}^2(\mathbb{R}) \setminus \lbrace \partial T \rbrace$ and look at the arrangement induced by $\mathcal{A}$ on $\mathbb{E}$; by abuse of notation we denote this arrangement by $\mathcal{A}$ as well.\\ 
Fix a line $\mathfrak{g}$ of the arrangement $\mathcal{A}$. Denote by $\mathcal{A}_{\mathfrak{g}}$ the set of all lines in $\mathcal{A}$ which are not parallel to $\mathfrak{g}$ and set $\mathcal{A}^\prime:=\mathcal{A}_{\mathfrak{g}} \cup \lbrace \mathfrak{g} \rbrace \subset \mathcal{A}$. Observe that if $\mathfrak{g}^\prime$ is a line of $\mathcal{A}$ parallel to $\mathfrak{g}$, then $\mathcal{A}_{\mathfrak{g}}=\mathcal{A}_{\mathfrak{g}^\prime}$. Assume that there is a vertex $v$ of $\mathcal{A}^\prime$ of weight three not lying on $\mathfrak{g}$. We may choose $v$ in such a way that there is a line $\mathfrak{g}_v \in \mathcal{A}^\prime$ passing through $v$ such that the bounded segment on $\mathfrak{g}_v$ reaching from $v$ to $\mathfrak{g} \cap \mathfrak{g}_v$ does not contain any other vertex of $\mathcal{A}^\prime$ of weight three.  We say that $v$ has \textit{distance} $k$ to $\mathfrak{g}$ if $k$ is minimal with the property that there is a line $\mathfrak{g}^\prime \in \mathcal{A}^\prime$ such that the interior of the bounded segment on $\mathfrak{g}^\prime$ reaching from $v$ to $\mathfrak{g}^\prime \cap \mathfrak{g}$ contains exactly $k$ vertices of $\mathcal{A}^\prime$ all of which have weight two.\\
Let us first consider the case where $v$ has distance zero to $\mathfrak{g}$. We will show that then there must be a vertex of $\mathcal{A}$ of weight three, contradicting our assumption on $\mathcal{A}$. There are two possibilities: either there are two lines $\mathfrak{g}_1, \mathfrak{g}_2$ passing through $v$ such that there is no vertex of weight two of $\mathcal{A}^\prime$ contained in the bounded segments reaching from $v$ to $\mathfrak{g} \cap \mathfrak{g}_1, \mathfrak{g} \cap \mathfrak{g}_2$ respectively, or there is only one such line. Consider the first possibility. Let $\mathfrak{g}_1, \mathfrak{g}_2$ be as above and denote by $\mathfrak{g}_3$ the third line passing through $v$. Similarly, denote by $\mathfrak{g}_4$ the third line passing through $\mathfrak{g} \cap \mathfrak{g}_2$ and assume that the bounded segment $s$ on $\mathfrak{g}_3$ reaching from $v$ to $\mathfrak{g} \cap \mathfrak{g}_3$ contains the vertex $\mathfrak{g}_3 \cap \mathfrak{g}_4$. Using the fact that there can be only finitely many lines of $\mathcal{A}^\prime$ passing through the bounded segment on $\mathfrak{g}$ reaching from $\mathfrak{g} \cap \mathfrak{g}_3$ to $\mathfrak{g} \cap \mathfrak{g}_2$, we see that $\mathcal{A}^\prime$ has a vertex $w$ of weight two contained in $s$. As by assumption every vertex of $\mathcal{A}$ has weight three, there must be a line $\mathfrak{g}_0 \in \mathcal{A}$ passing through $w$ which is parallel to $\mathfrak{g}$. But then there is a vertex of weight two of $\mathcal{A}$ contained in the line $\mathfrak{g}_2$, a contradiction. Now we deal with the second possibility. Denote the three lines passing through $v$ again by $\mathfrak{g}_1, \mathfrak{g}_2, \mathfrak{g}_3$. We may assume that the bounded segment on $\mathfrak{g}_1$ reaching from $v$ to $\mathfrak{g} \cap \mathfrak{g}_1$ does not contain any vertices of $\mathcal{A}^\prime$. Moreover, we may assume that $  \mathfrak{g} \cap \mathfrak{g}_1$ is not contained in the bounded segment on $\mathfrak{g}$ reaching from $\mathfrak{g} \cap \mathfrak{g}_2$ to $\mathfrak{g} \cap \mathfrak{g}_3$ and that $  \mathfrak{g} \cap \mathfrak{g}_3$ is not contained in the bounded segment on $\mathfrak{g}$ reaching from $\mathfrak{g} \cap \mathfrak{g}_1$ to $\mathfrak{g} \cap \mathfrak{g}_2$.  Again, using the fact that there are only finitely many lines of $\mathcal{A}^\prime$ passing through the bounded segment on $\mathfrak{g}$ reaching from $\mathfrak{g} \cap \mathfrak{g}_1$ to $\mathfrak{g} \cap \mathfrak{g}_2$, we conclude that $\mathcal{A}^\prime$ must have a vertex $w^\prime$ of weight two contained in the bounded segment on line $\mathfrak{g}_2$ reaching from $v$ to $ \mathfrak{g}  \cap \mathfrak{g}_2$. Thus, there must be a line in $\mathcal{A}$ parallel to $\mathfrak{g}$ and passing through $w^\prime$. But then $\mathcal{A}$ must have a vertex of weight two, contradicting our assumption.\\
Now assume that the distance from $v$ to $\mathfrak{g}$ is greater than zero and call it $k$. Let $\mathfrak{g}^\prime \in \mathcal{A}^\prime$ be a line passing through $v$ and containing exactly $k$ vertices $v_0, ..., v_{k-1}$ of weight two of $\mathcal{A^\prime}$ between $v$ and $ \mathfrak{g}  \cap \mathfrak{g}^\prime$. Without loss of generality we may assume that $v_{k-1}$ is closest to $v$. As the arrangement $\mathcal{A}$ has only vertices of weight three we conclude that there must be a line $\mathfrak{g}^{\prime \prime}$ parallel to $\mathfrak{g}$ and passing through $v_{k-1}$. Now consider the arrangement $\mathcal{A}^{\prime  \prime}:=\mathcal{A}_\mathfrak{g} \cup \lbrace \mathfrak{g}^{\prime \prime} \rbrace \subset \mathcal{A}$. We observe that $\mathcal{A}^\prime$ and $\mathcal{A}^{\prime  \prime}$ differ only by one line not belonging to $\mathcal{A}_\mathfrak{g}$, hence $v$ is also a vertex of $\mathcal{A}^{\prime  \prime}$ and its distance to $\mathfrak{g}^{\prime  \prime}$ is zero. As above this implies that $\mathcal{A}$ has a vertex of weight two, which is impossible by assumption. Hence $\mathcal{A}^\prime$ has no vertex $v$ as above: every vertex of $\mathcal{A}^\prime$ of weight three must lie on $\mathfrak{g}$. This shows that $\mathcal{A}$ must contain infinitely many lines which are parallel to $\mathfrak{g}$. But then $F$ cannot be irreducible as it must contain a linear factor corresponding to the infinitely many lines of $\mathcal{A}$ parallel to $\mathfrak{g}$.  
\end{proof}

\begin{corollary}
There is no affine Tits arrangement $\mathcal{A}$ of rank three such that $\mathcal{A}^*$ is contained in the locus of an irreducible homogeneous polynomial of degree three.
\end{corollary}

\begin{proof}
Consider an arrangement $\mathcal{A}$ of lines in the real projective plane such that $\mathcal{A}^* \subset V(P) \subset (\mathbb{P}^2(\mathbb{R}))^*$ for some irreducible $P \in \mathbb{R}[x,y,z]$ with $\deg(P)=3$. Let $v \in \mathbb{P}^2(\mathbb{R})$ be an arbitrary vertex of $\mathcal{A}$. Then in the dual setting $v^*$ is given by a line and the weight of $v$ is bounded by $|v^* \cap V(P)|$. Bézout's theorem gives $|v^* \cap V(P)| \leq \deg(v^*) \cdot \deg(P)=3$. Now if $\mathcal{A}$ is simplicial and affine then by Lemma \ref{4 mal 3  lemma} each vertex of $\mathcal{A}$ has weight exactly three. But then by Lemma \ref{a2_affin lemma} it follows that $P$ cannot be irreducible. 
\end{proof}

\begin{remark} \label{bezoutremark}
a) If one drops the assumption on $\mathcal{A}$ to be affine in Corollary 1, then the proof above shows that there are possible candidates for (spherical) Tits arrangements $\mathcal{A}$ such that $\mathcal{A}^* \subset V(P)$: namely all spherical arrangements having only vertices of weight two or three. Since these are precisely the arrangements $A(6,1), A(7,1)$ and the near pencils with at most four lines, we will not investigate this further. \\
b) If $\mathcal{A}^* \subset V(P)$ for some possibly reducible polynomial $P$, we may still apply Bézout's theorem to conclude the following: suppose that $P$ is a product of three linear factors. Then $\mathcal{A}$ has at most three vertices of weight possibly bigger than three and all other vertices have weight bounded by three. If on the other hand $P$ is the product of an irreducible quadratic factor and a linear factor, then $\mathcal{A}$ has at most one vertex of weight possibly bigger than three while all other vertices have weight bounded by three.  
\end{remark}

It remains to consider the possibility that $\mathcal{A}^*$ is contained in the locus of a cubic homogeneous polynomial having an irreducible quadratic factor. As preparation, we introduce some more notation.

\begin{definition} \label{consec def}
a) Let $\sigma$ be an irreducible conic in $\mathbb{P}^2(\mathbb{R})$ and consider a subset $M \subset \sigma$. There exists a projectivity $\Psi$ such that $\Psi(\sigma)$ is given by the polynomial $P:=x^2+y^2-z^2$ and is thus contained entirely in the affine $z=1$ patch of $\mathbb{P}^2(\mathbb{R})$. We say that $p_1,...,p_k \in M$ are \textit{consecutive with respect to} $\Psi$, if for any $1 \leq i \leq k-1$ it is true that one of the segments on $\Psi(\sigma)$ bounded by $\Psi(p_i),\Psi(p_{i+1})$ contains no other point of $\Psi(M)$.    \\
b) Consider the map $\phi: \mathbb{R}^3 \times \mathbb{R}^3 \longrightarrow \mathbb{R}^3$ sending $v_1,v_2 \in \mathbb{R}^3$ to their vector product $v_1 \times v_2$. This induces a map $\psi: \left( \mathbb{P}^2(\mathbb{R}) \times \mathbb{P}^2(\mathbb{R}) \right) \setminus \Delta \longrightarrow (\mathbb{P}^2(\mathbb{R}))^*$, where $\Delta:=\lbrace (x,x) \mid x \in \mathbb{P}^2(\mathbb{R}) \rbrace$. 
By a slight abuse of notation, we write $\psi(v_1,v_2) =v_1 \times v_2 \in (\mathbb{P}^2(\mathbb{R}))^*$ for two different projective points $v_1, v_2 \in \mathbb{P}^2(\mathbb{R})$. Observe that for $p,q \in (\mathbb{P}^2(\mathbb{R}))^*$ the vector product $p \times q$ gives the vertex in $\mathbb{P}^2(\mathbb{R})$ obtained as the intersection of the dual lines $p^*,q^*$. Similarly, if $v, v^\prime$ are two points in $\mathbb{P}^2(\mathbb{R})$, then the vector product $v \times v^\prime$ gives the point in $(\mathbb{P}^2(\mathbb{R}))^*$ which is dual to the line passing through $v$ and $v^\prime$. 
\end{definition}

Now we can prove the following statement (compare also \cite[Thm.\ 3.6]{p-C10b}, where case c) of the following proposition is examined for spherical Tits arrangements).

\begin{proposition} \label{conic line prop}
Suppose that $\mathcal{A}$ is an affine rank three Tits arrangement and assume that $\mathcal{A}^* \subset \sigma \cup \mathfrak{l}$ for some irreducible conic $\sigma \subset (\mathbb{P}^2(\mathbb{R}))^*$ and an arbitrary line $\mathfrak{l}\subset (\mathbb{P}^2(\mathbb{R}))^*$. Then the following statements hold:
\begin{enumerate}
\item[a)] $|\mathcal{A}^* \cap \sigma|=\infty$, unless $\mathcal{A}$ is a near pencil.
\item[b)] $|\mathcal{A}^* \cap \mathfrak{l}|=\infty$ and $(\partial T)^* \in \mathfrak{l}$.
\item[c)] If $|\sigma \cap \mathfrak{l}|=0$ then $\mathcal{A}$ is a near pencil.
\item[d)] If $|\sigma \cap \mathfrak{l}|=1$ then $\sigma \cap \mathfrak{l}=(\partial T)^*$, unless $\mathcal{A}$ is a near pencil.
\item[e)] If $|\sigma \cap \mathfrak{l}|=2$ then $\mathcal{A}$ is a near pencil.
\end{enumerate}
\end{proposition}

\begin{proof}
a) Define $\mathcal{B}:=\mathcal{A} \cap \sigma^*$ and suppose that $|\mathcal{B}|<\infty$. Since $\mathcal{A}$ is affine and hence necessarily infinite, we set $L:=\mathcal{A} \cap \mathfrak{l}^*$ and conclude that $|L|=\infty$. So we have $\mathcal{A}=\mathcal{B} \cup L$ and it is easy to see that we find a line in $\mathcal{B}$ containing a segment bounded by two vertices of weight two. By Lemma \ref{near pencil lemma} we conclude that $\mathcal{A}$ is a near pencil. \\
b) If $\mathcal{A}$ is a near pencil then both statements are easily seen to be true. So we may assume that $\mathcal{A}$ is not a near pencil. We show that the second statement is a consequence of the first. So suppose that $|\mathcal{A}^* \cap \mathfrak{l}|=\infty$ and assume that $(\partial T)^* \notin \mathfrak{l}$. Dualizing we obtain that the point $\mathfrak{l}^*$ does not lie on the line $\partial T$. Hence $\mathfrak{l}^*$ lies in $T$ and there are infinitely many lines of $\mathcal{A}$ passing through $\mathfrak{l}^*$. But since $\mathcal{A}$ is locally finite in $T$ this is impossible. So it suffices to prove that $|\mathcal{A}^* \cap \mathfrak{l}|=\infty$. We show that $|\mathcal{A}^* \cap \mathfrak{l}|<\infty$ gives a contradiction: fix some $q \in \sigma \cap \mathcal{A}^*$ and consider the pencil $\mathcal{P}_q$ of lines $\mathfrak{l}_{q, q^\prime} \subset (\mathbb{P}^2(\mathbb{R}))^*$ passing through $q$ and $q^\prime  \in  (\sigma \cap \mathcal{A}^*) \setminus \lbrace q \rbrace$. By part a) it follows that $|\mathcal{A}^* \cap \sigma|=\infty$, since by assumption $\mathcal{A}$ is not a near pencil. In particular $|\mathcal{P}_q|=\infty$. Hence there must be a pair of neighbouring lines $\mathfrak{l}_{q, q^\prime}, \mathfrak{l}_{q, q^{\prime\prime}} \in \mathcal{P}_q$ whose intersections with $\mathfrak{l}$ are both not contained in $\mathcal{A}^*$. This is true because by assumption there are only finitely many points in $\mathcal{A}^* \cap \mathfrak{l}$. But this means that the line $q^* \in \mathcal{A}$ must contain a segment bounded by two vertices of weight two, which by Lemma \ref{near pencil lemma} implies that $\mathcal{A}$ is a near pencil. This is the desired contradiction. \\
c) Since $\sigma \cap \mathfrak{l}=\emptyset$ we may use part b) to conclude that $(\partial T)^* \notin \sigma$. But then it follows that $|\mathcal{A}^* \cap \sigma|<\infty$, since points of $\mathcal{A}^*$ may accumulate only in a neighbourhood of $(\partial T)^*$ (because $\mathcal{A}$ is locally finite in $T$). Now by part a) it follows that $\mathcal{A}$ is a near pencil. \\
d) By part b) we already know that $(\partial T)^* \in \mathfrak{l}$. Assume that $(\partial T)^* \notin \sigma$. Then it follows that $|\mathcal{A}^* \cap \sigma|<\infty$, because points of $\mathcal{A}^*$ may accumulate only in a neighbourhood of $(\partial T)^*$. Hence we may use part a) to conclude that $\mathcal{A}$ must be a near pencil.   \\
e) After applying a projectivity as in part a) of Definition \ref{consec def}, we may assume that $\sigma=V(P)$ where $P:=x^2+y^2-z^2$. So $\sigma$ is contained entirely in the affine $z=1$ patch of $(\mathbb{P}^2(\mathbb{R}))^*$. We write $\sigma^\prime$ for the conic in $\mathbb{P}^2(\mathbb{R})$ defined by the same polynomial. \\
Suppose that $\mathcal{A}$ is not a near pencil. As points of $\mathcal{A}^*$ may accumulate only in a neighbourhood of $(\partial T)^*$, we have $(\partial T)^* \in \sigma \cap \mathfrak{l}$. Observe that for $p=(a:b:1) \in \sigma \cap \mathcal{A}^* \subset (\mathbb{P}^2(\mathbb{R}))^*$ the corresponding dual line $p^*$ is the tangent to $\sigma^\prime$ at the point $\left( -a:-b:1 \right) \in \mathbb{P}^2(\mathbb{R})$. In particular, if $(\partial T)^*=\left( x:y:1 \right)$, this implies that there is a sequence of tangent lines to $\sigma^\prime$ converging towards the tangent line at the point $\left( -x:-y:1 \right)$, and this tangent line is precisely $\partial T$. It remains to identify the dual lines $q^*$ corresponding to $q \in \mathfrak{l} \cap \mathcal{A}^*$. We may assume without loss of generality that in the $z=1$ patch of $(\mathbb{P}^2(\mathbb{R}))^*$ the line $\mathfrak{l}$ is given by the equation $y=\lambda$ for some $0 \leq \lambda<1$. Hence any $q \in \mathfrak{l}$ will have homogeneous coordinates $q=(x_0:\lambda:1)$. So if $\lambda>0$, the equation of the	dual line $q^*$ in the $z=1$ patch of $\mathbb{P}^2(\mathbb{R})$ will be $y=-\frac{x_0 \cdot x }{\lambda} - \frac{1}{\lambda}$; if on the other hand $\lambda=0$, then the equation of $q^*$ will be $x=-\frac{1}{x_0}$. Hence if $\lambda>0$, then all lines pass through the point $(0:-\frac{1}{\lambda}:1)$ which implies that $\mathfrak{l}^*=(0:-\frac{1}{\lambda}:1)$; if $\lambda=0$, then all lines pass through $\mathfrak{l}^*=(0:1:0)$. This shows that $\mathfrak{l}^* \notin \sigma^\prime$. Since $(\partial T)^* \in \mathfrak{l}$ we conclude that $\mathfrak{l}^* \in \partial T$. Now we take $\partial T$ as line at infinity. Doing so, we obtain $\mathcal{A}$ as union of tangent lines to a parabola together with infinitely many parallel lines each of which being non-parallel to the symmetry axis of the parabola. But then $\mathcal{A}$ is not simplicial. 
\end{proof}

The following lemma will be the key to proving the main theorem.

\begin{lemma} \label{inzidenz_lemma}
Let $\sigma$ be an irreducible conic together with a projectivity $\Psi$ as in part a) of Definition \ref{consec def}. Assume that $\mathell$ is a line touching $\sigma$. If $\mathcal{A}$ is an irreducible affine rank three Tits arrangement such that $\mathcal{A}^* \subset \sigma \cup \mathell$, then $\mathcal{A}$ is determined by specifying four points on $\sigma$ which are consecutive with respect to $\Psi$. More precisely, if $p_{-1},p_0, p_{1}, p_{2}, p_{3},p_{4} \in \mathcal{A}^* \cap \sigma$ are six consecutive points (with respect to $\Psi$), then we have the following formulas for $p_{-1}$ and $p_{4}$ in terms of $p_0,..., p_3$: \begin{align}
p_{4}&= \left( p_0 \times \left( \mathell^* \times \left( p_{1} \times p_{3} \right)  \right) \right) \times \left( p_{1} \times \left( \mathell^* \times \left( p_{2} \times p_{3} \right)  \right) \right), \\
p_{-1}&= \left( p_{2} \times \left( \mathell^* \times \left( p_{0} \times p_{1} \right)  \right) \right) \times \left( p_{3} \times \left( \mathell^* \times \left( p_{0} \times p_{2} \right)  \right) \right).
\end{align} 
\end{lemma}

\begin{proof}
Denote by $L_1, L_2 \subset \mathcal{A}$ the set of lines corresponding to elements in $\mathcal{A}^* \cap \sigma, \mathcal{A}^* \cap \mathell$  respectively. Observe that every $\mathfrak{h} \in L_2$ passes through the point $\mathell^*$ while no line belonging to $L_1$ passes through $\mathell^*$: if $\mathell^* \in \mathfrak{g}$ and $\mathfrak{g}^* \in \sigma$ for some $\mathfrak{g}$, then $\mathfrak{g}^* \in \mathell \cap \sigma=\lbrace (\partial T)^* \rbrace $, by part e) of Proposition \ref{conic line prop}. As $\mathcal{A}$ is thin by definition, we conclude that $\mathfrak{g} \notin \mathcal{A}$. \\
Note also that every vertex of weight two of $\mathcal{A}$ must lie on a line belonging to $L_2$. Indeed, assume there was a vertex $v$ of weight two such that $v = \mathfrak{g} \cap \mathfrak{g}^\prime$ for some $\mathfrak{g}, \mathfrak{g}^\prime \in L_1$. As $\mathcal{A}^* \subset \sigma \cup \mathell$ and because no line belonging to $L_1$ passes through $\mathell^*$, we may use part b) of Remark \ref{bezoutremark} to conclude that every neighbour of $v$ has weight bounded by three. But then by Lemma \ref{near pencil lemma} every neighbour of $v$ has weight precisely three, because $\mathcal{A}$ was assumed to be irreducible. By Lemma \ref{4 mal 3  lemma} we obtain that $\mathcal{A}$ is spherical, a contradiction. In particular, it follows that for every vertex $v^\prime$ obtained as intersection of elements in $L_1$ there is a line $\mathfrak{h} \in L_2$ passing through $v^\prime$. Also, every vertex of weight two is a neighbour of $\mathell^*$.  \\
These conditions already suffice to prove the claim. Let $p_0, p_1, p_2, p_3 \in \mathcal{A}^* \cap \sigma$ be four consecutive points (with respect to $\Psi$). We need to construct the points $p_{-1}, p_4 \in \mathcal{A}^*\cap \sigma$ such that both $p_{-1}, p_0, p_1, p_2$ and $p_1, p_2, p_3, p_4$ are consecutive (with respect to $\Psi$). By symmetry, it suffices to construct $p_4$. For this, denote the line corresponding to $p_i$ by $\mathfrak{g}_i$ and let $\mathfrak{h}$ be the line passing through the vertices $\mathell^*,\mathfrak{g}_1 \cap \mathfrak{g}_3$. Similarly, denote by $\mathfrak{h}^\prime$ the line passing through the vertices $\mathell^*, \mathfrak{g}_2 \cap \mathfrak{g}_3$. Then $\mathfrak{g}_4$ is the line passing through the vertices $\mathfrak{g}_0 \cap \mathfrak{h}, \mathfrak{g}_1 \cap \mathfrak{h}^\prime$. From this, one reads off that (1) holds. This completes the proof.  
\end{proof}

\begin{remark}
Let $P$ be a homogeneous cubic polynomial having an irreducible quadratic factor. If $\mathcal{A}$ is an irreducible spherical Tits arrangement such that $\mathcal{A}^* \subset V(P)$, then one may use part b) of Remark \ref{bezoutremark} to conclude that there are two possibilities for $\mathcal{A}$: either $\mathcal{A}$ is the arrangement $A(7,1)$ or $\mathcal{A}$ belongs to the infinite family $\mathcal{R}(1)$.   
\end{remark}

Now we can construct the arrangement of type $\tilde{A}^0_2$ and prove that up to projectivity it is the only irreducible affine rank three Tits arrangement whose dual point set is contained in the locus of a cubic polynomial having an irreducible quadratic factor: 

\begin{proposition} \label{conic_line}
Up to projectivity, there is only one irreducible affine rank three Tits arrangement $\mathcal{A}$ such that $\mathcal{A}^*$ is contained in the locus of a cubic polynomial $P$ having an irreducible quadratic factor. The arrangement $\mathcal{A}$ may be defined by the following set of dual points: \begin{align*}
\mathcal{A}^* = \left\lbrace \left( k:\frac{k(k-1)}{2}:1 \right), \left(1:\frac{k}{2}:0 \right) \mid k \in \mathbb{Z} \right\rbrace.
\end{align*}   
\end{proposition}

\begin{proof}
Let $\mathell \subset (\mathbb{P}^2(\mathbb{R}))^*$ be the line corresponding to the linear factor of $P$ and let $\sigma \subset (\mathbb{P}^2(\mathbb{R}))^*$ be the irreducible conic corresponding to the quadratic factor of $P$. We then have $\mathcal{A}^* \subset \sigma \cup \mathell \subset (\mathbb{P}^2(\mathbb{R}))^*$ and by Proposition \ref{conic line prop} we may assume that $\mathell$ touches $\sigma$ at the point $(\partial T)^*$. \\
Let $p_1, p_2, p_3, p_4 \in \mathcal{A}^* \cap  \sigma$ be four consecutive points (with respect to some projectivity $\Psi$). 
After a change of coordinates we may assume that \begin{align*}
(\partial T)^*&=\left(0:1:0\right), p_2=\left(1:0:1\right),   \\
p_3&=\left(2:1:1\right), p_4=\left(3:3:1\right). 
\end{align*}
We then have $p_1=\left(x:y:z \right)$ for some $x,y,z \in \mathbb{R}$. Now consider the vertices $v:= p_2 \times p_3, v^\prime:= p_1 \times p_4 \in \mathbb{P}^2(\mathbb{R})$ and let $\mathfrak{g} \subset \mathbb{P}^2(\mathbb{R})$ be the line passing through $v$ and $v^\prime$. Then by (the proof of) Lemma \ref{inzidenz_lemma} we know that $\mathfrak{g} \in \mathcal{A}$ and that $\mathfrak{g}$ passes through the vertex $\mathell^* $. As $\mathell^* \in \partial T$, we may write $\mathell^*=(a:0:b)$ for certain $a, b \in \mathbb{R}$. In order to prove the statement we will distinguish four cases. \\
 
Case 1. Assume that $x=y=0$. This implies that $p_1=\left( 0:0:1 \right)$. We claim that $\mathell^*=\left(0:0:1 \right)$. To see this write $\mathell^*=\left(a:0:b \right)$ for some $a,b \in \mathbb{R}$ as above. The fact that $\mathfrak{g}$ passes through $\mathell^*$ implies that $a=0$ and therefore we have $\mathell^*=\left( 0:0:1 \right)$. \\
Now consider the projectivity $\Phi: (\mathbb{P}^2(\mathbb{R}))^* \longrightarrow (\mathbb{P}^2(\mathbb{R}))^*$ taking the point $p_i$ to $p_{i+1}$ for $1 \leq i \leq 4$. 
We obtain $\mathcal{A}^* \cap \sigma=\left\lbrace \Phi^k(p_1) \mid k \in \mathbb{Z} \right\rbrace= \left\lbrace \left( k:\frac{k(k-1)}{2}:1 \right) \mid k \in \mathbb{Z} \right\rbrace$, using Lemma \ref{inzidenz_lemma} and induction. Observe that the lines of $\mathcal{A}$ corresponding to points in $\mathcal{A}^* \cap \mathell	$ are exactly the lines passing through $\mathell^*$ and a vertex of the form $p \times p^\prime$ for $p, p^\prime \in \mathcal{A}^* \cap \sigma$ (see the proof of Lemma \ref{inzidenz_lemma}). We conclude that $\mathcal{A}^* \cap \mathell = \left\lbrace \left(1:\frac{k}{2}:0 \right) \mid k \in \mathbb{Z} \right\rbrace$. It is now easy to check that $\mathcal{A}^*=\left\lbrace \left( k:\frac{k(k-1)}{2}:1 \right), \left(1:\frac{k}{2}:0 \right) \mid k \in \mathbb{Z} \right\rbrace$ defines an irreducible affine Tits arrangement. \\

Case 2. Assume that $x \neq 0 $ and $y = 0$. Then we may assume that $p_1=\left( 1:0:z \right)$. Write $\mathell^*=\left(a:0:b \right)$ for $a,b \in \mathbb{R}$. The fact that $\mathfrak{g}$ passes through $\mathell^*$ implies that $a \neq 0$. Thus, we may assume that $\mathell^*=\left( 1:0:b \right)$. It follows that $z=\frac{b+4}{3}$ and therefore $p_1=\left(1:0:\frac{b+4}{3} \right)$. Observe that the five given points $(\partial T)^*, p_1,p_2, p_3, p_4$ on $\sigma$ determine its equation. Using this together with Lemma \ref{inzidenz_lemma}, the condition $ p_5 \in \sigma $ implies that $b \in \lbrace -1,- \frac{3}{2}, - \frac{7}{3},-3 \rbrace$. As $p_0,p_5 \neq p_i$ for $1 \leq i \leq 4$, we conclude that $b \in \lbrace -1,- \frac{3}{2}, -3 \rbrace$ is impossible. In the remaining case $b=-\frac{7}{3}$, we observe that the conic $\sigma$ may be defined by the polynomial $f=-\frac{10}{3} X^2 + 2 XY + \frac{28}{3} XZ - \frac{10}{3} YZ - 6 Z^2$. By assumption, we know that the line $\mathell$ touches $\sigma$ at the point $(\partial T)^*$. Thus, as $\mathell^*=\left(1:0:-\frac{7}{3} \right)$, there exists $0 \neq \lambda \in \mathbb{R}$ such that the following equations are satisfied: \begin{align*}
 1&=\lambda \ \frac{\partial f}{\partial_X}\big\rvert_{(\partial T^*)}, \\
 0&=\lambda \ \frac{\partial f}{\partial_Y}\big\rvert_{(\partial T^*)}, \\
 -\frac{7}{3}&=\lambda \ \frac{\partial f}{\partial_Z}\big\rvert_{(\partial T^*)}. 
\end{align*}
The first equation gives $\lambda=\frac{1}{2}$. But then the third equation reads $-\frac{7}{3}=-\frac{5}{3}$.  
This contradiction shows that Case 2 cannot occur. \\
  
Case 3. Assume that $x=0$ and $y \neq 0$. Then without loss of generality, we may assume that $p_1=\left( 0:1:z \right)$. Again, we write $\mathell^*=\left(a:0:b \right)$ for suitable $a,b \in \mathbb{R}$ and as $\mathfrak{g}$ passes through $\mathell^*$, we obtain $a \neq 0$. Thus, we may assume that $\mathell^*=(1:0:b)$, leading to $z= - \frac{b + 3}{3}$. We conclude that $p_1= (0:1:- \frac{b + 3}{3})$. The relation $p_5 \in \sigma$ gives $b \in \lbrace -3, -1 \rbrace$. As $p_5 \neq p_i$ for $1 \leq i \leq 4$, we conclude that this is impossible. \\

Case 4. Assume that both $x \neq 0$ and $y \neq 0$. Then we may suppose that $p_1=(1:y:z)$. Write $\mathell^*=(a:0:b)$ for suitable $a, b \in \mathbb{R}$. As before, by considering the line $\mathfrak{g}$, we conclude that $-3 z a-3 a y-b y+4 a+b=0$. Suppose that $a=0$. Then without loss of generality $b=1$ and we have $y=1$, in particular $p_1=(1:1:z)$. As $p_5 \in \sigma$, we conclude that $z \in \lbrace \frac{1}{3}, \frac{1}{2} \rbrace$. Again, this is not possible because $p_0,p_5 \neq p_i$ for $1 \leq i \leq 4$. \\
 Hence, we may assume that $a=1$. In particular, we have $z = \frac{4}{3} - \frac{b (y-1)}{3} - y$ and $p_1=(1:y:\frac{4}{3} - \frac{b (y-1)}{3} - y)$.\\
 Suppose that $b \neq -3$. Using the condition $p_5 \in \sigma$, we compute that $y \in \left\lbrace 1, \frac{-3 b^2 - 10 b - 7}{2 (b + 3)},
 \frac{2 b^2 + 5 b + 3}{2 (b^2 + 3 b + 3)} \right\rbrace$. As $p_1 \neq p_4$, we can exclude the case $y=1$. \\
 Assume that $y=\frac{2 b^2 + 5 b + 3}{2 (b^2 + 3 b + 3)}$. Then we obtain $p_1=p_5$, a contradiction. So we necessarily have $y= \frac{-3 b^2 - 10 b - 7}{2 (b + 3)}$. In particular, this implies that $p_1=\left(1:\frac{-3 b^2 - 10 b - 7}{2 (b + 3)}:\frac{b^2+4b+5}{2} \right)$. Therefore, the conic $\sigma$ may be defined by the polynomial $f:=(b-1) X^2+2 XY-(b-7) XZ-2(b+4) YZ-6 Z^2$. To see this, one only has to check that $f(p_i)=0$ for $1\leq i \leq 5$. The line $\mathell$ touches $\sigma$ at the point $(\partial T)^*=\left(0:1:0 \right)$. Therefore, as $\mathell^*=\left(1:0:b \right)$, we know that there exists $0 \neq \lambda \in \mathbb{R}$ such that the following equations hold: \begin{align*} 
 1&=\lambda \ \frac{\partial f}{\partial_X}\big\rvert_{(\partial T^*)} , \\
 0&=\lambda \ \frac{\partial f}{\partial_Y}\big\rvert_{(\partial T^*)} , \\
 b&=\lambda \ \frac{\partial f}{\partial_Z}\big\rvert_{(\partial T^*)}. 
\end{align*} 
The first equation gives $\lambda=\frac{1}{2}$. Thus, the third equation yields $b=-2$ and we obtain $p_1=\left(1:\frac{1}{2}:\frac{1}{2} \right)=\left(2:1:1\right)=p_3$, a contradiction. \\
It remains to consider the case $b=-3$. Then we have $\mathell^*=\left(1:0-3 \right)$ and $p_1=\left(1:y:\frac{1}{3} \right)$. Clearly, we have $y \neq 1$ because $p_1 \neq p_4$. Then Lemma \ref{inzidenz_lemma} yields $p_5=\left(3:3:1\right)=p_4$, another contradiction. This completes the proof.   
\end{proof}

We obtain the following Corollary:

\begin{corollary}
There are irreducible affine Tits arrangements which are not locally spherical.
\end{corollary}

\begin{proof}
This follows from Proposition \ref{conic_line}. The arrangement constructed there is such an example: the vertex $\mathell^*$ is incident with infinitely many lines of $\mathcal{A}$.
\end{proof}

Finally, using Proposition 2, Corollary 1, Proposition 3, and Proposition 4, we obtain the promised main theorem:

\begin{theorem}
Let $\mathcal{A}$ be an affine rank three Tits arrangement such that $\mathcal{A}^*$ is contained in the locus of a homogeneous polynomial of degree three. Then up to projectivity $\mathcal{A}$ is either a near pencil, an arrangement of type $\tilde{A_2}$, or it is an arrangement of type $\tilde{A^0_2}$.
\end{theorem}

\section{Open questions and related problems}
In this section we want to point out some possibly interesting related problems. First, we ask if there exists an affine rank three Tits arrangement $\mathcal{A}$ (viewed as arrangement of lines in the real projective plane) such that $\mathcal{A}^*$ is contained entirely in the locus of an irreducible homogeneous polynomial: 

\begin{theoremc}
Is there some irreducible homogeneous polynomial $P \in \mathbb{R}[x,y,z]$ such that $\mathcal{A}^* \subset V(P)$ for a suitable irreducible affine rank three Tits arrangement $\mathcal{A}$?
\end{theoremc}

Observe that given a Tits arrangement $\mathcal{A}$ and an irreducible homogeneous polynomial $P$ of degree $d$ such that $\mathcal{A}^* \subset V(P)$, it follows immediately that $\mathcal{A}$ is locally spherical. Indeed, suppose there was a vertex $v$ of $\mathcal{A}$ such that infinitely many lines of $\mathcal{A}$ pass through $v$. Then after dualizing it follows that infinitely many points of $\mathcal{A}^*$ lie on the line $v^*$. But since by assumption $\mathcal{A}^* \subset V(P)$, it follows that infinitely many points lie on the intersection $V(P) \cap v^*$. But Bézout's theorem tells that $|V(P) \cap v^*| \leq d \cdot 1 =d < \infty$, because $P$ was assumed to be irreducible and hence $v^*$ cannot be a component of $V(P)$. This contradiction shows that $\mathcal{A}$ must be locally spherical. 
\\

This leads to the next problem. Are there other examples of irreducible affine rank three Tits arrangements which are not locally spherical?

\begin{theoremc}
Classify (up to projectivities) all irreducible affine rank three Tits arrangements $\mathcal{A}$ which are not locally spherical. 
\end{theoremc} 

Observe that if $\mathcal{A}$ is not locally spherical, then by Lemma \ref{nur ein dicker vertex lemma} there is precisely one vertex $v$ on the boundary of the Tits cone $T$. In particular, it follows that for every line $\mathfrak{l} \neq v^*$ we have $|\mathcal{A}^* \cap \mathfrak{l}|< \infty$. If in addition we know that $\mathcal{A}^* \subset V(P)$ for some homogeneous polynomial $P$ of degree $d$, then by Bézout's theorem the last inequality can be strengthened to $$|\mathcal{A}^* \cap \mathfrak{l}| \leq |V(P) \cap \mathfrak{l}| \leq d$$ for every line $\mathfrak{l} \neq v^*$ which is not a component of $V(P)$.

We close this section by proposing the following final problem which is probably the most difficult:
\begin{theoremc}
Classify (up to projectivities) all affine rank three Tits arrangements $\mathcal{A}$ such that $\mathcal{A}^* \subset V(P)$ for some homogeneous polynomial $P \in \mathbb{R}[x,y,z]$.
\end{theoremc}
A solution to the last problem seems to be an important step towards a classification of all affine rank three Tits arrangements. Indeed, if $\mathcal{A}$ is such an arrangement and if $\mathcal{A}= \bigcup_{i \in I} L_i$ for some finite index set $I$ and sets of mutually parallel lines $L_i, i \in I$, then $\mathcal{A}^*$ is contained in the locus of a polynomial $P$ of degree $|I|$: the polynomial $P$ is a product of linear factors corresponding to the sets $L_i, i \in I$.
For example, affine Tits arrangements coming from Nichols algebras of diagonal type are always of this type.

Even if we enlarge $\mathcal{A}$ by finitely many countable subsets of tangent lines to certain conics, we still find a polynomial $P^\prime$ such that the enlarged arrangement is contained in the locus of $P^\prime$. The polynomial $P^\prime$ may be taken as the product of $P$ together with the irreducible quadratic polynomials defining the (dual) conics in question. This gives the impression that the class of rank three affine Tits arrangements lying on the locus of some polynomial is rather large, as demonstrated by the fact that only usage of at most quadratic polynomials already leads to nontrivial considerations. It may even be conjectured that for every irreducible rank three affine Tits arrangement $\mathcal{B}$ there is a certain polynomial $Q$ such that $\mathcal{B}^* \subset V(Q)$. If this is true, then clearly a solution to Problem 3 amounts to a complete classification of affine rank three Tits arrangements.       

\def\cprime{$'$}
\providecommand{\bysame}{\leavevmode\hbox to3em{\hrulefill}\thinspace}
\providecommand{\MR}{\relax\ifhmode\unskip\space\fi MR }
\providecommand{\MRhref}[2]{%
  \href{http://www.ams.org/mathscinet-getitem?mr=#1}{#2}
}
\providecommand{\href}[2]{#2}

\end{document}